\documentclass[final,1p,times]{elsarticle}

\usepackage{amssymb}
\usepackage{amsthm}
\usepackage{mathrsfs}

\usepackage{mathptmx}
\usepackage{amsmath,amsfonts}
\usepackage{url,hyperref,multirow}
\usepackage{float}
\usepackage{color}
\usepackage{epsfig,bm,caption2,epstopdf}

\usepackage{setspace}
\usepackage{exscale}
\usepackage{relsize}
\usepackage{booktabs}
\usepackage{ifpdf}

\usepackage{dsfont}

\newtheorem{theorem}{Theorem}[section]
\newtheorem{remark}{Remark}[section]
\newtheorem{corollary}{Corollary}[section]
\newtheorem{lemma}{Lemma}[section]
\newtheorem{example}{Example}[section]

\newtheorem{assumption}{Assumption}

\journal{}

\begin{document}
	
	\begin{frontmatter}
		
		
		
		\title{Pointwise-in-time convergence analysis of an Alikhanov scheme for a 2D nonlinear subdiffusion equation}
		
		
		
		\author[OUC]{Chang Hou  }
		\ead{1546758808@qq.com}

		\author[OUC,LOUC]{Hu Chen }
		\ead{chenhu@ouc.edu.cn} 
		
		\author[OUC]{Jian Wang\corref{wang}}
		\ead{pdejwang@ouc.edu.cn} \cortext[wang]{Corresponding author.}
		
		%
		

		\address[OUC]{School of Mathematical Sciences, Ocean University of China, Qingdao 266100, China}
         \address[LOUC]{Laboratory of Marine Mathematics, Ocean University of China, Qingdao 266100, China}
		\begin{abstract}
			In this paper, we discretize the Caputo time derivative of order $\alpha\in(0,1)$ using the Alikhanov scheme on a quasi-graded temporal mesh, and employ the Newton linearization method to approximate the nonlinear term. This yields a linearized fully discrete scheme for the two-dimensional nonlinear
 time fractional subdiffusion equation with weakly singular solutions. For the purpose of conducting a pointwise convergence analysis using the comparison principle, we develop a new stability result.  The global $L^2$-norm convergence order is $\text{min}\{\alpha r,2\}$, and the local $L^2$-norm convergence order is $\text{min}\{r,2\}$ under appropriate conditions and assumptions. Ultimately, the rates of convergence demonstrated by the numerical experiments serve to validate the analytical outcomes.
		\end{abstract}
		
		
	\end{frontmatter}
	
	
	
	\section{Introduction}
In this work, the following fractional parabolic problem is considered:
\begin{align}
	\label{eq1.1}
	& D_{t}^{\alpha }u(\bm{x},t)+Lu(\bm{x},t)=f(u(\bm{x},t)),\ (\bm{x},t)\in \Omega\times \left(0,T\right],\\
	&u(\bm{x},0)=u_0(\bm{x}),\ \bm{x}\in \Omega,\\
	&u(\bm{x},t)=0,\ (\bm{x},t)\in \partial\Omega\times \left[0,T\right],
	\label{eq1.3}
\end{align}
where the Caputo fractional derivative $D_{t}^{\alpha }$ with respect to $t$ is defined by  $ D_{t}^{\alpha}u(\bm{x},t)=\frac{1}{\Gamma (1-{{\alpha }})}\int_{0}^{t}{{{(t-s)}^{-{{\alpha }}}}\frac{\partial u(\bm{x},s)}{\partial s}ds}$, $\Omega\subset \mathbb{R}^2$ is an
open bounded spatial domain $(0,L_1)\times(0,L_2)$ with a boundary $\partial\Omega$, $u_0(\bm{x})$ is a continuous function in $\bar\Omega$. The operator $Lu$ denotes $-\nu\Delta u(\nu>0)$.

In recent years, fractional calculus has been incorporated into models because of its ability to capture long-term memory effects\cite{MR3051755,MR4485803,MR3231767}.  Liao et al. \cite{MR3790081} obtained the temporal global convergence of the L1 scheme for the linear reaction-subdiffusion equation by applying the Gronwall inequality. The linear reaction-subdiffusion problem serves as a particular case, i.e., $f(u)=cu$ in \eqref{eq1.1}.
Kopteva \cite{MR4129005} discretized the Caputo derivative term using the L1 scheme on quasi-graded temporal meshes with degree of grading $r$ and obtained a fully discrete scheme for the nonlinear fractional parabolic equation, which has a local convergence order $\text{min}\{r,2-\alpha\}$ and a global convergence order $\text{min}\{\alpha r,2-\alpha\}$ by employing the comparison principle. However, the nonlinear term $f(u)$ must satisfy the one-sided Lipschitz condition. Under the assumption that \( f \in C^1(\mathbb{R}) \), Li et al. \cite{MR4729054} linearized the nonlinear term and employed mathematical induction to derive the pointwise error estimates.

Compared with the L1 scheme, the Alikhanov scheme  exhibits a higher order of convergence. This scheme has been widely applied\cite{MR4453273,MR4863143}. The L-type discrete formula involves a longer computational time, hence some researchers have proposed fast algorithms to overcome the issue of computational complexity. Cen and Wang  \cite{MR4372681} employed Alikhanov scheme and the two-grid technique to discretize nonlinear subdiffusion problems satisfying the Lipschitz condition. However, these three works impose stronger assumptions on the smoothness of the solution and are unable to address problems where the solution exhibits a weak singularity at the initial time. Chen et al. \cite{MR4338910} used an improved discrete fractional Gronwall inequality. It has been proven that on graded meshes, the error in the $H^1$-norm of the Alikhanov scheme is robust with respect to the order $\alpha$ for linear fractional subdiffusion equation. An et al. \cite{MR4434375} applied  the L1 scheme and Alikhanov scheme on graded meshes in temporal direction, and utilized FEM in space for linear fractional superdiffusion equation. Then they obtain that the $H^1$-norm analysis of error estimates is $\alpha$-robust. Liu et al. \cite{MR4544723} used the Newton linearization method to approximate the nonlinear term, Alikhanov scheme with sum-of-exponentials approximation to accelerate the computation of Caputo derivative and Galerkin finite element method in spatial discretization. They obtained the global error estimate of the scheme for nonlinear subdiffusion equations on the nonuniform mesh.  Liao et al. \cite{MR4270344} conducted an error convolution structure analysis for the Alikhanov scheme and derived a sharp \(L^2\)-norm global error estimate \(\text{max}\|e^n\|\le C\tau^{\min\{2,\alpha r\}}\) on the nonuniform mesh for linear reaction-subdiffusion problems, where \(\tau\) is the maximum step size. In \cite{MR4092654}, it was demonstrated that the proposed Alikhanov scheme in the paper preserves the maximum principle for the Allen-Cahn equation. Huang et al. \cite{MR4402734} proved a new  discrete fractional Gronwall inequality, and utilized it to demonstrate that the error of the Allen-Cahn equation discretized by the Alikhanov-FEM scheme on a graded mesh is $\alpha$-robust in the $H^1$-norm.  Although the high convergence order of the Alikhanov scheme has attracted the interest of researchers, the above studies have only analyzed the global error estimates. In 2020, Kopteva and Meng \cite{MR4085134} considered the discretization of linear fractional parabolic equations using the L1 and  Alikhanov schemes. They employed barrier functions to derive a crucial stability result, which subsequently enabled them to obtain pointwise error estimates in the temporal direction for the L1 and Alikhanov  schemes.  Currently, there is a lack of pointwise error estimates for nonlinear subdiffusion equations discretized via the Alikhanov scheme. This work is designed to fill this gap.

Since the convergence orders of the L1 scheme and the Alikhanov scheme are identical on a uniform mesh, our primary interest lies in non-uniform meshes. The analytical work in this paper is conducted under the assumption  $f(u)\in C^2(\mathbb{R})$ and Assumption \ref{assumption1.1}. Given the Theorem 3.1 of \cite{MR3742688}, inequality \eqref{eq1.4} is reasonable for $k=1$. The results similar to inequality \eqref{eq1.4} are proven in the Lemmas 2.4 and 2.6 of \cite{MR4122156} for Allen-Cahn equation. We first discretize \eqref{eq1.1}--\eqref{eq1.3} using the Alikhanov scheme in time and a standard second-order approximation in space. Then we utilize a Newton linearized method for the nonlinear term. Our key contribution lies in obtaining a novel stability result for Alikhanov scheme under some conditions:
 \begin{align}
 	\label{keyeq}
 	\begin{cases}
 		(\delta_t^{\alpha,*}-\lambda_1)  v^j-\lambda_2 v^{j-1}\lesssim (\tau_1/t_j)^{\gamma+1}\\
 		\forall\ j\ge 1, v^0=0
 	\end{cases}\Rightarrow  v^j\lesssim V(j,\gamma):=\tau_1 t_j^{\alpha-1}(\tau_1/t_j)^{\min\{0,\gamma\}}.
 \end{align}
 In 2025, Kopteva et al. proposed an inequality analogous to \eqref{keyeq}, with their research focusing on the L1 scheme\cite{2025arXiv250612954K}. Finally we provide the pointwise-in-time error estimates and stability analysis. We find that a much milder grading $r\ge 2$ is sufficient to achieve the optimal convergence order 2 when $t$ is away from 0. This  provides sharper error estimates for such equations.
\begin{assumption}
		\label{assumption1.1}
The solution $u$ of (1.1)-(1.3) satisfies  $u(\cdot,t)\in C^4(\bar\Omega)$ for $t\ge 0$, and
	\begin{align}	
		\label{eq1.4}
		\left| \frac{{{\partial }^{k}}u(\bm{x},t)}{\partial {{t}^{k}}} \right|+\left|\frac{{{\partial }^{k}}\Delta u(\bm{x},t)}{\partial {{t}^{k}}}\right|\le C(1+{{t}^{{{\alpha }}-k}}),\ k=0,1,2,3,
	\end{align}
for $\bm{x}\in \bar{\Omega}$, and $0<t\le T$.
\end{assumption}
This paper consists of the following sections. In Section 2, we construct a fully discrete scheme. In Section 3, the corresponding truncation errors are estimated. In Section 4, the comparison principle and a key stability result for Alikhanov scheme are proved. In Section 5, the convergence and stability of the fully discrete, linearized scheme are analyzed, and we present the sharp pointwise-in-time error estimates. In Section 6, the theoretical analysis is validated through numerical examples.

\textbf{Notation.}
In this paper, the positive constant $C$ may vary from line to line, and is independent of the time steps and the spatial steps. $a\lesssim b$ indicates that there exists a positive constant $C$ such that $a\le Cb$. Besides, $a\simeq b$ means that $a\lesssim b$ and $b\lesssim a$.

\section{Fully discrete scheme on a quasi-graded mesh}
The domain \(\Omega\) is uniformly partitioned into a grid which consists of \(M_1 \times M_2\) cells. The spatial step sizes are defined as \(h_1 = L_1/M_1\) and \(h_2 = L_2/M_2\). The set of all grid points is denoted by $\Omega' = \{(x_i, y_j)= (i h_1, j h_2),\ i = 0, \dots, M_1 \ and\ j = 0, \dots, M_2\}$, with \(\Omega_h = \Omega' \cap \Omega\) representing the interior nodes and \(\partial\Omega_h = \Omega' \cap \partial\Omega\) the boundary nodes.
The time interval \([0, T]\) is divided using a quasi-graded mesh \(\{t_n\}_{n=0}^N\) with the degree of grading \(r \ge 1\). The time steps \(\tau_j = t_j - t_{j-1}\) may be non-uniform, with \(\tau = \max_j \tau_j\) denoting the maximum step size. $\rho$ denotes $\max_{1\le j\le N-1}{\tau_j/\tau_{j+1}}$. The temporal mesh satisfies
\begin{align}
	\tau_1\simeq N^{-r},\ t_n\simeq\tau_1n^r,\ \tau_n\simeq\tau_1^{1/r}t_n^{1-1/r}.
\end{align}

At $t_n^*=t_n-(1-\sigma)\tau_n$, ($1\ge\sigma\ge0$), applying the Alikhanov's formula, $D_t^\alpha u$  can be approximated as
\begin{align*}
	D_t^\alpha u(x_i,y_j,t_n^*)&=\frac{1}{\Gamma(1-\alpha)}\int_{0}^{t_n^*}{(t_n^*-s)^{-\alpha}\partial_su(x_i,y_j,s)ds}\\
	&=\frac{1}{\Gamma(1-\alpha)}\sum_{k=1}^{n-1}\int_{t_{k-1}}^{t_{k}}{(t_n^*-s)^{-\alpha}\partial_su(x_i,y_j,s)ds} \\
&\quad +\frac{1}{\Gamma(1-\alpha)}\int_{t_{n-1}}^{t_n^*}{(t_n^*-s)^{-\alpha}\partial_su(x_i,y_j,s)ds}\\
	&\approx \frac{1}{\Gamma(1-\alpha)}\sum_{k=1}^{n-1}\int_{t_{k-1}}^{t_{k}}{(t_n^*-s)^{-\alpha}\partial_s\Pi_{2,k} u(x_i,y_j,s)ds}\\
&\quad +\frac{1}{\Gamma(1-\alpha)}\int_{t_{n-1}}^{t_n^*}{(t_n^*-s)^{-\alpha}\partial_s\Pi_{1,n}u(x_i,y_j,s)ds}\\
	&=\sum_{k=0}^{n}p_{n,k}u(x_i,y_j,t_k)\\
	&:=\delta_t^{\alpha,*}u_{i,j}^n,
\end{align*}
where $\Pi_{2,j}u(s)$ is the standard Lagrange interpolation polynomial passes through points $(t_{j-1},u(t_{j-1}) )$, $(t_j,u(t_j))$, and $ (t_{j+1},u(t_{j+1}))$, while the interpolation nodes for $\Pi_{1,n}u(s)$ are $t_{n-1}$ and $ t_n$. Using linear interpolation, the following approximation can be obtained
\begin{equation}
	\begin{split}
		u(x_i,y_j,t_n^*)&=u(x_i,y_j,\sigma t_n+(1-\sigma )t_{n-1})\\
		&\approx \sigma u(x_i,y_j,t_n)+(1-\sigma ) u(x_i,y_j,t_{n-1})\\
		&:=u_{i,j}^{n,*}.
	\end{split}
\end{equation}

Defining $u_{i,j}^{n}=u(x_i,y_j,t_n)$, the previous equation \eqref{eq1.1}-\eqref{eq1.3} can be transformed into
\begin{align}
	\label{eq2.1}
	\delta_{t}^{\alpha,*}u_{i,j}^n+L_hu_{i,j}^{n,*}&=f(u_{i,j}^{n-1})+\sigma f'(u_{i,j}^{n-1})(u_{i,j}^n-u_{i,j}^{n-1})\notag\\&+(r_1)_{i,j}^n+(r_2)_{i,j}^n+(r_3)_{i,j}^n,	(x_i,y_j)\in \Omega_h, t_n\in(0,T]\\
	u_{i,j}^0&=u_0(x_i,y_j), (x_i,y_j)\in \Omega_h\\
	u_{i,j}^n&=0, (x_i,y_j)\in \partial\Omega_h,  t_n\in[0,T],
	\label{eq2.3}
\end{align}
where $(r_1)_{i,j}^n=\delta_{t}^{\alpha,*}u_{i,j}^n-D_t^\alpha u(x_i,y_j,t_n^*)$, $(r_2)_{i,j}^n=L_hu_{i,j}^{n,*}-Lu(x_i,y_j,t_n^*)$ and $(r_3)_{i,j}^n=f(u(x_i,y_j,t_n^*))-f(u_{i,j}^{n-1})-\sigma f'(u_{i,j}^{n-1})(u_{i,j}^n-u_{i,j}^{n-1})$. $L_hu$ is a standard second-order approximation of $Lu$, i.e., $u_{xx}(x_i,y_j,t_n)\approx \frac{1}{h_1^2}(u_{i+1,j}^n-2u_{i,j}^n+u_{i-1,j}^n)$ and $u_{yy}(x_i,y_j,t_n)\approx\frac{1}{h_2^2}(u_{i,j+1}^n-2u_{i,j}^n+u_{i,j-1}^n)$.

Leaving out the truncation errors, a fully discrete scheme can be constructed at points $(x_i,y_j,t_n^*)$, i.e.,
\begin{align}
	\label{eq2.4}
	\delta_{t}^{\alpha,*}U_{i,j}^n+L_hU_{i,j}^{n,*}&=f(U_{i,j}^{n-1})+\sigma f'(U_{i,j}^{n-1})(U_{i,j}^n-U_{i,j}^{n-1}),	(x_i,y_j)\in \Omega_h, t_n\in(0,T]\\
	U_{i,j}^0&=u_0(x_i,y_j), (x_i,y_j)\in \Omega_h\\
	U_{i,j}^n&=0, (x_i,y_j)\in \partial\Omega_h,  t_n\in[0,T],
	\label{eq2.6}
\end{align}
where $U_{i,j}^n$ is a grid function at $(x_i,y_j,t_n)$.
\begin{lemma}\cite[Sect. 2]{MR3936261}
	\label{lemma2.1}
	$\delta_t^{\alpha,*}U^m$ can be represented as
	$\delta_t^{\alpha,*}U^m=\sum_{j=1}^{m}{g_{m-1,j-1}(U^j-U^{j-1})}$ as well. The coefficient  $g_{0,0}=\tau_1^{-1}a_{0,0}$ and when $k\ge 1$ coefficients
	\begin{equation}
		\begin{split}
			g_{k,j}=\begin{cases}
				\tau_{j+1}^{-1}(a_{k,0}-b_{k,0}),\quad &if\ j=0,\\
				\tau_{j+1}^{-1}(a_{k,j}+b_{k,j-1}-b_{k,j}),\quad &if\ 1\le j\le k-1,\\
				\tau_{j+1}^{-1}(a_{k,k}+b_{k,k-1}),\quad &if\ j=k.\\
			\end{cases}
		\end{split}
	\end{equation}
	with $a_{k,k}=\frac{\sigma^{1-\alpha}}{\Gamma(2-\alpha)}\tau_{k+1}^{1-\alpha}$ for $k\ge0$, $a_{k,j}=\frac{1}{\Gamma(1-\alpha)}\int_{t_j}^{t_{j+1}}{(t_{k+1}^*-\eta)^{-\alpha}d\eta}$,\\ $b_{k,j}=\frac{1}{\Gamma(1-\alpha)}\frac{2}{t_{j+2}-t_j}\int_{t_j}^{t_{j+1}}{(t_{k+1}^*-\eta)^{-\alpha}(\eta-t_{j+1/2})d\eta}$ for $k\ge 1$, $0\le j\le k-1$.
\end{lemma}

\begin{lemma}
	\label{lemma2.2}
If $\sigma=1-\frac{\alpha}{2}$, and $\rho_{j-1}\ge\rho_j\ge0.618$ ($\rho_j=\tau_{j+1}/\tau_j$) for $2\le j\le N-1$, the coefficients in Lemma \ref{lemma2.1} have the following properties: \\
P1: $g_{m-1,j}>g_{m-1,j-1}>0$, for $1\le j\le m-1\le N-1$.\\
P2: $(2\sigma-1)g_{m-1,m-1}> \sigma g_{m-1,m-2}$, for $2\le m\le N$.\\
P3: There exists a constant $\pi_A>0$, such that
\begin{align}
	\label{eq2.7}
	g_{m-1,j-1}\ge \frac{1}{\pi_A\tau_j}\int_{t_{j-1}}^{t_j}{\frac{(t_m-s)^{-\alpha}}{\Gamma(1-\alpha)}ds},\ for\ 1\le j\le m\le N.
\end{align}
P4: There exists a constant $\rho>0$, such that $\tau_j/\tau_{j+1}\le \rho$, for $1\le j\le N-1$.
\end{lemma}
\begin{proof}
 For the detailed proofs that properties P1 and P2 hold, please refer to Lemma 4 of \cite{MR3936261}. Specifically, the conditions for P1 is weaker, requiring only $1\ge\sigma\ge1-\frac{\alpha}{2}$, and $\rho_{j-1}\ge\rho_j\ge0.4656$ ($\rho_j=\tau_{j+1}/\tau_j$) for $j\ge2$.
 According to the condition $\rho_j\ge 0.618$, it follows that $\tau_{j}/ \tau_{j+1}\le (0.618)^{-1}\le 7/4$ for $j\ge 1$. Therefore, the assumption M1 in Theorem 2.2 of \cite{MR4270344} is satisfied. There exists a constant $\pi_A=\frac{11}{4}$ such that the inequality \eqref{eq2.7} holds.
\end{proof}
\section{Truncation error estimates}
Suppose $\sigma=1-\frac{\alpha}{2}$, the approximation of the Caputo derivative is given by the Alikhanov scheme. Next, we estimate the truncation error for the case $\sigma=1-\frac{\alpha}{2}$.
\begin{lemma}(\cite{MR4085134}, Theorem 4.4)
	\label{lemma2.3}
	For $(x_i,y_j)\in \Omega_h$ and $1\le n\le N$, we have
	\begin{align*}
		\left|(r_1)_{i,j}^n\right|\le C(\tau_1/t_n)^{\hat\gamma+1},\ where\  \hat\gamma+1:=min\{\alpha+1,(3-\alpha)/r\}.
	\end{align*}
\end{lemma}
\begin{lemma}
	\label{lemma2.4}
	For $(x_i,y_j)\in \Omega_h$ and $1\le n\le N$, it holds that
	$\left|(r_3)_{i,j}^n\right|\le C\tau_n^2t_n^{\alpha-2}$.
\end{lemma}
\begin{proof}
		Case A: $n>1$.\\
	According to the Taylor expansion, we can see that
	\begin{align*}
		u(x_i,y_j,t_n^*)&=u(x_i,y_j,t_n)+\partial_tu(x_i,y_j,t_n)(\sigma-1)(t_n-t_{n-1})\\
           &\quad \quad +\frac{1}{2}\partial_{tt}u(x_i,y_j,\xi_1)\left[(\sigma-1)(t_n-t_{n-1})\right]^2,\ \xi_1\in(t_n^*,t_n),\\
		u(x_i,y_j,t_n^*)&=u(x_i,y_j,t_{n-1})+\partial_tu(x_i,y_j,t_{n-1})\sigma(t_n-t_{n-1}) \\
&\quad \quad+\frac{1}{2}\partial_{tt}u(x_i,y_j,\xi_2)\left[\sigma(t_n-t_{n-1})\right]^2,\ \xi_2\in(t_{n-1},t_n^*).
	\end{align*}
	Then, using Lagrange mean value theorem, we can obtain that
	\begin{align*}
		\left|u(x_i,y_j,t_n^*)-u_{i,j}^{n,*}\right|&\le C\tau_n^2\left|\partial_{tt}u(x_i,y_j,\xi)\right|\\
		&\le C\tau_n^2(1+t_{n-1}^{\alpha-2})\\
		&\le C\tau_n^2(1+t_{n}^{\alpha-2})\\
		&\le  C\tau_n^2t_{n}^{\alpha-2},
	\end{align*}
	where $\xi\in(t_{n-1},t_n)$. We have used $t_n\simeq t_{n-1}$ for $n\ge2$ in the penultimate inequality.
	
	Case B: $n=1$.
	\begin{align*}
		\left|u(x_i,y_j,t_1^*)-u_{i,j}^{1,*}\right|&=\left|\sigma\left(u(x_i,y_j,t_1^*)-u_{i,j}^1\right)+(1-\sigma)\left(u(x_i,y_j,t_1^*)-u_{i,j}^0\right)\right|\\
		&=\left|\sigma\left(\int_{t_1}^{t_1^*}{\partial_su(x_i,y_j,s)ds}\right)+(1-\sigma)\left(\int_{t_0}^{t_1^*}{\partial_su(x_i,y_j,s)ds}\right)\right|\\
		&\le \sigma\int_{t_1^*}^{t_1}{\left|\partial_su(x_i,y_j,s)\right|ds}+(1-\sigma)\int_{t_0}^{t_1^*}{\left|\partial_su(x_i,y_j,s)\right|ds}\\
		&\le C\left[\sigma\int_{t_1^*}^{t_1}{1+s^{\alpha-1}ds}+(1-\sigma)\int_{t_0}^{t_1^*}{1+s^{\alpha-1}ds}\right]\\
		&\le  C(t_1+ t_1^\alpha/\alpha)\le C\tau_1^2t_1^{\alpha-2}.
	\end{align*}
	In the penultimate inequality, we have used $t_1^*\le t_1$.
	
	To sum up, we have $\left|u(x_i,y_j,t_n^*)-u_{i,j}^{n,*}\right|\le C\tau_n^2t_n^{\alpha-2}$, for $n\ge 1$.	
	Taking into account that $u$ is bounded, it follows that, for $n\ge 1$,
	\begin{align*}
		\left|(r_3)_{i,j}^n\right|&\le \left|f(u(x_i,y_j,t_n^*))-f(u_{i,j}^{n,*})\right|+\left|f(u_{i,j}^{n,*})-f(u_{i,j}^{n-1})-\sigma f'(u_{i,j}^{n-1})(u_{i,j}^n-u_{i,j}^{n-1})\right|\\
		&\le f'(\xi_3)\left|u(x_i,y_j,t_n^*)-u_{i,j}^{n,*}\right|+\left|\frac{1}{2}f''(\xi_4)\sigma^2(u_{i,j}^n-u_{i,j}^{n-1})^2\right|\\
		&\le C\tau_n^2t_n^{\alpha-2}+C\left|u_{i,j}^n-u_{i,j}^{n-1}\right|^2\\
		&\le C\tau_n^2t_n^{\alpha-2},		
	\end{align*}
	where $\xi_3$ is between $u(x_i,y_j,t_n^*)$ and $u_{i,j}^{n,*}$, $\xi_4$ is between $u_{i,j}^{n-1}$ and $u_{i,j}^{n,*}$; the last inequality is obtained by
	\begin{align*}
		|u^n-u^{n-1}|&=\begin{cases}
		|u'(\xi_5)|\tau_n\le C\tau_n(1+t_{n-1}^{\alpha-1})\le C\tau_n(1+t_{n}^{\alpha-1}),\ for\ n\ge2\\
		|\int_{t_{0}}^{t_1}{u'(s)ds}|\le C\int_{t_{0}}^{t_1}{1+s^{\alpha-1}ds}\le Ct_1^{\alpha},\ for\ n=1
		\end{cases}\\
		&\le C\tau_nt_{n}^{\alpha-1},
	\end{align*}
	where $\xi_5$ is between $t_{n-1}$ and $t_{n}$.
	
\end{proof}

\begin{lemma}
	\label{lemma2.5}
	$\left|(r_2)_{i,j}^n\right|\le C(\tau_n^2t_n^{\alpha-2}+h_1^2+h_2^2)$, 	for $(x_i,y_j)\in \Omega_h$ and $1\le n\le N$.
\end{lemma}
\begin{proof}
The truncation error can be decomposed into two components, hence we have
\begin{align*}
	\left|(r_2)_{i,j}^n\right|&= \left|L_hu_{i,j}^{n,*}-Lu_{i,j}^{n,*}+Lu_{i,j}^{n,*}-Lu(x_i,y_j,t_n^*)\right|\\
	&\le
	\left|L_hu_{i,j}^{n,*}-Lu_{i,j}^{n,*}\right|+	\left|Lu_{i,j}^{n,*}-Lu(x_i,y_j,t_n^*)\right|\\
	&\le C(\tau_n^2t_n^{\alpha-2}+h_1^2+h_2^2).
\end{align*}
	The proof follows a similar procedure as that of Lemma \ref{lemma2.4}. Additionally, the assumption $\left|\frac{{{\partial }^{k}}\Delta u(\bm{x},t)}{\partial {{t}^{k}}}\right|\le C(1+{{t}^{{{\alpha }}-k}})$, for $k=0,1,2,3,$ is required.
\end{proof}
\section{Stability result for Alikhanov scheme}
\begin{lemma}
	Set $1\ge \sigma\ge\frac{1}{2}$. Suppose that property P1 holds, and    $(\delta_t^{\alpha,*}-\lambda_1)w^j-\lambda_2w^{j-1}\ge 0$, $w^0\ge 0$ for all $j\ge1$, where $\lambda_1,\lambda_2\ge 0$. Let the temporal mesh satisfy $\lambda_1\tau^\alpha\le \frac{1}{2\Gamma(2-\alpha)}$. We have $w^j\ge0 $ for $ j\ge0$.
\end{lemma}
\begin{proof}
	$w^0\ge0$ is obvious. Using P1, we know that $\delta_t^{\alpha,*}w^n$ can be expressed as $\delta_t^{\alpha,*}w^n=\sum_{j=0}^{n}{p_{n,j}w^j}$ for $1\le n\le N$, where $p_{n,n}>0$ and $p_{n,j}<0$ when $j<n$. If $w^j\ge 0$ holds for $j=0,1,...,k-1$, we have $(p_{k,k}-\lambda_1) w^k\ge \sum_{j=0}^{k-1}{\left|p_{k,j}\right|w^j}+\lambda_2w^{k-1}\ge 0$. This together with $p_{k,k}-\lambda_1> 0$, gives $w^k\ge0$. One can infer that $p_{k,k}\ge \tau_k^{-\alpha}\frac{\sigma^{1-\alpha}}{\Gamma(2-\alpha)}\ge \tau^{-\alpha}\frac{\sigma^{1-\alpha}}{\Gamma(2-\alpha)}$. Owing to $\sigma^{1-\alpha}> \frac{1}{2}$, one has $p_{k,k}-\lambda_1>0$. The induction is thus completed.
\end{proof}

\begin{corollary}
	\label{corollary3.1}
	Let the temporal mesh satisfy $1\ge \sigma\ge\frac{1}{2}$, $\lambda_1\tau^\alpha\le \frac{1}{2\Gamma(2-\alpha)}$ and property P1 hold. If $A^0\le B^0$ and $(\delta_t^{\alpha,*}-\lambda_1)A^j-\lambda_2A^{j-1}\le (\delta_t^{\alpha,*}-\lambda_1)B^j-\lambda_2B^{j-1}$ for all $j\ge1$, where $\lambda_1,\lambda_2\ge 0$, we have $A^j\le B^j$, for $j\ge0$.
\end{corollary}

\begin{lemma}
	\label{lemma3.1}
Set $1\ge \sigma\ge\frac{1}{2}$. Let $\lambda_1\lambda_2\ge 0$ and $\lambda_1+\lambda_2> 0$. Select an arbitrary positive constant $c_0$ such that $c_0<\frac{1}{2}\left[2(\lambda_1+\lambda_2)\Gamma(2-\alpha)\right]^{-1/\alpha}$ and an arbitrary \(\tau = \max_{1\le j\le N} \tau_j\)  such that $\tau\le c_0/2$. Then it follows that,
for any $\hat t=t_m\in\{t_j\}_{j=0}^N$, there exists $\{B^{j}\}_{j=0}^N$, satisfying $B^j=0\ when\ j\le m$, $0\le B^j\lesssim 1$, and \begin{align}\label{eq3.1}
	(\delta_t^{\alpha,*}-\lambda_1)B^j-\lambda_2B^{j-1}\gtrsim
	\begin{cases}
		0, \quad t_j^*< t_m+c_0,\\
		1, \quad t_j^*\ge t_m+c_0,
	\end{cases}
\end{align}
when $j\ge 1$.
\end{lemma}
\begin{proof}
The proof process is inspired by Lemma 3.2 in \cite{MR4129005}.	
	Set $B(t)=\sum_{k=0}^{K}{\bar{c}^kB_k(t)}$, where $B_k(t)=\max\{0,t-q_k\}$, $q_0=\hat t$,
	\begin{align*}
	 q_k=\begin{cases}
		\max\{t_j|t_j\in [	q_{k-1}+c_0-\tau,q_{k-1}+c_0]\},\quad &if\  T\ge q_{k-1}+c_0-\tau,\\
		\forall q \in [	q_{k-1}+c_0-\tau,q_{k-1}+c_0], &otherwise.
	\end{cases}
\end{align*}
	 For $T> \hat t+c_0$, the selection of $K$ must ensure that $T\in(q_{K+1},q_{K+2}]$ which implies that $K\le \lceil {2(T-\hat t)/c_0} \rceil-1\le \lceil {2T/c_0} \rceil-1$. Specially, if $T\le \hat t+c_0$, set $K=0$.
  Denote $\lambda^*=\lambda_1+\lambda_2$. Then it is clear that $D_t^{\alpha}B_k(t_j^{*})-\lambda^* B_k(t_j)=0$ when $q_k\ge t_j^{*}$. Subsequently, we shall consider the scenario where $q_k<t_j^{*}$, i.e. $q_k\le t_{j-1}$. Using the definition of $D_t^{\alpha}$, the result is as follows:
  \begin{align*}
  	D_t^{\alpha}B_k(t_j^{*})-\lambda^* B_k(t_j)&=\frac{1}{\Gamma(1-\alpha)}\int_{q_k}^{t_j^*}{(t_j^*-s)^{-\alpha}\partial_s(s-q_k)ds}-\lambda^*(t_j-q_k)\\
  	&=(t_j^{*}-q_k)^{1-\alpha}\left[\frac{1}{\Gamma(2-\alpha)}-\lambda^*\frac{t_j-q_k}{t_j^{*}-q_k}(t_j^{*}-q_k)^\alpha\right]\\
  	&\ge(t_j^{*}-q_k)^{1-\alpha}\left[\frac{1}{\Gamma(2-\alpha)}-2\lambda^*(t_j^{*}-q_k)^\alpha\right].
  \end{align*}
  The last inequality employs
  \begin{align*}
  	\frac{t_j-q_k}{t_j^{*}-q_k}
  	&=\left(1+\frac{t_j^*-t_j}{t_j-q_k}\right)^{-1}\\
  	&=\left[1-\frac{(1-\sigma)\tau_j}{t_j-q_k}\right]^{-1}\\
  	&\le\left[1-\frac{\tau_j}{2(t_j-q_k)}\right]^{-1}\\
  	&\le  2.
  \end{align*}
 If $t_j^*-q_k>0$, and $t_j^*-q_k\le 2c_0$, the expression on the right-hand side of the inequality is strictly greater than zero. Therefore,
  \begin{align*}
  	D_t^{\alpha}B_k(t_j^{*})-\lambda^* B_k(t_j)
  	&\ge \begin{cases}
  		0, \quad t_j^*\in (0,q_{k+1}),\\
  		\min\{(\frac{c_0}{2})^{1-\alpha}/\Gamma(2-\alpha)-\lambda^* c_0,(2c_0)^{1-\alpha}/\Gamma(2-\alpha)-4\lambda^* c_0\}, \quad t_j^*\in [q_{k+1},q_{k+2}],\\
  		-\left|T^{1-\alpha}/\Gamma(2-\alpha)-2\lambda^* T\right|, \quad t_j^*\in (q_{k+2},T].
  	\end{cases}\\
  	&:=
  	\begin{cases}
  		0, \quad t_j^*\in (0,q_{k+1}),\\
  		S_1, \quad t_j^*\in [q_{k+1},q_{k+2}],\\
  		-S_2, \quad t_j^*\in (q_{k+2},T].
  	\end{cases}
  \end{align*}
  Let $\bar{c}=\frac{S_1+S_2}{S_1}$, and we have
  \begin{align*}
  	\bar{c}^lS_1-\sum_{j=0}^{l-1}{\bar{c}^jS_2}=\bar{c}^lS_1-\frac{1-\bar{c}^l}{1-\bar{c}}S_2=S_1,\ \forall l\ge 1.
  \end{align*}
 Thus,
  \begin{equation}
  	\begin{split}
  		D_t^{\alpha}B(t_j^*)-\lambda^* B(t_j)&=\sum_{k=0}^{K}{\bar{c}^k\left(D_t^{\alpha}B_k(t_j^*)-\lambda^* B_k(t_j)\right)}\\
  		&\ge
  		\begin{cases}
  			0, \quad t_j^*\in (0,q_{1}),\\
  			S_1, \quad t_j^*\in [q_{1},T].\\
  		\end{cases}
  		\label{result1}
  	\end{split}
  \end{equation}

If $t_j\le q_n$, it is obvious that $\delta_t^{\alpha,*}B_n^j=D_t^\alpha B_n(t_j^*)=0$.	Suppose that there exists $t_k\in\{t_j\}_{j=1}^{N}$ such that $q_n=t_k\neq0$ and $t_j\ge t_{k+1}$, we have
	\begin{align*}
		\delta_t^{\alpha,*}B_n^j&=\frac{1}{\Gamma(1-\alpha)}\int_{t_{k-1}}^{t_k}{(t_j^*-s)^{-\alpha}(\Pi_{2,k} B_n(s))'ds}+\frac{1}{\Gamma(1-\alpha)}\int_{t_{k}}^{t_j^*}{(t_j^*-s)^{-\alpha}ds}\\
		&=\frac{-1}{\Gamma(1-\alpha)}\int_{t_{k-1}}^{t_k}{\alpha(t_j^*-s)^{-\alpha-1}(\Pi_{2,k} B_n(s))ds}+\frac{1}{\Gamma(1-\alpha)}\int_{t_{k}}^{t_j^*}{(t_j^*-s)^{-\alpha}ds}\\
		&\ge \frac{1}{\Gamma(1-\alpha)}\int_{t_{k}}^{t_j^*}{(t_j^*-s)^{-\alpha}ds}=D_t^\alpha B_n(t_j^*).
	\end{align*}
	 Specially, $\delta_t^{\alpha,*}B_n^j= D_t^\alpha B_n(t_j^*)$ for any $j\ge 1$, if $q_n=t_k=0$. In conclusion, $\delta_t^{\alpha,*}B_n^j\ge D_t^\alpha B_n(t_j^*)$ holds true for $j\ge 1$ and $0\le n\le K$.
	
	Owing to $B_k(t_j)\ge B_k(t_{j-1})$, for $1\le j\le N$, it holds that
	\begin{align*}
		\delta_t^{\alpha,*} B^j-\lambda_1B(t_{j})-\lambda_2B(t_{j-1})
		&=\sum_{k=0}^{K}{\bar{c}^k\left(\delta_t^{\alpha,*} B_k^j-\lambda_1B_k(t_{j})-\lambda_2B_k(t_{j-1})\right)}\\
		&\ge \sum_{k=0}^{K}{\bar{c}^k\left(\delta_t^{\alpha,*} B_k^j-\lambda^* B_k(t_{j})\right)}\\
		&\ge
		\sum_{k=0}^{K}{\bar{c}^k\left(D_t^\alpha B_k(t_j^*)-\lambda^* B_k(t_{j})\right)}.
	\end{align*}
It is obvious that $t_j^*\ge q_1$ when $t_j^*\ge t_m+c_0$.	By combining \eqref{result1}, the inequality \eqref{eq3.1}  holds.
\end{proof}

\begin{theorem}
	\label{lemma3.3}
	Assume that $\sigma=1-\frac{\alpha}{2}$, $\tau$ is sufficiently small and property P1 holds. If $\lambda_1\ge 0$ and $\lambda_2\ge 0$, we can derive the following important inequality: \\
	(i) For $1\le r \le (3-\alpha)/\alpha$, $\gamma\in R$ and $\gamma \neq 0$, it holds that
	\begin{align}
		\label{eq3.2}
		\begin{cases}
			(\delta_t^{\alpha,*}-\lambda_1)  v^j-\lambda_2 v^{j-1}\lesssim (\tau_1/t_j)^{\gamma+1}\\
			\forall\ j\ge 1, v^0=0
		\end{cases}\Rightarrow  v^j\lesssim V(j,\gamma):= \tau_1 t_j^{\alpha-1}(\tau_1/t_j)^{\min\{0,\gamma\}}.
	\end{align}
(ii) If $r\ge 1$ and $\gamma\le \alpha-1$, the above result holds as well.
\end{theorem}
\begin{proof}
	(i)The Theorem 4.2 of reference \cite{MR4085134} has already reached the conclusion when $\lambda_1+\lambda_2=0$ that
	\begin{align*}
			\begin{cases}
			\delta_t^{\alpha,*} \xi_\gamma^j= (\tau_1/t_j)^{\gamma+1}\\
			\forall\ j\ge 1, \xi_\gamma^0=0
		\end{cases}\Rightarrow \left| \xi_\gamma^j\right|\le C_1 \tau_1t_j^{\alpha-1}(\tau_1/t_j)^{\min\{\gamma,0\}}=C_1\tau_1^\alpha(\tau_1/t_j)^{1+\min\{\gamma,0\}-\alpha}\ for\ j\ge 1.
	\end{align*}
	 Next, we prove that the result \eqref{eq3.2} can be obtained when $\lambda_1+\lambda_2\neq 0$. Define $\gamma^*=\min\{\gamma,0\}-\alpha<0$. We can obtain two functions $\xi_\gamma$, $\xi_{\gamma^*}$ that satisfy the following inequalities:
	 \begin{gather*}
	 	0\le \xi_\gamma^j\le C_1 \tau_1^\alpha(\tau_1/t_j)^{1+\min\{\gamma,0\}-\alpha},\\
	 	0\le \xi_{\gamma^*}^j\le C_2 \tau_1^\alpha(\tau_1/t_j)^{1+\gamma^*-\alpha}.
	 \end{gather*}
	 Using $t_j\simeq t_{j-1}$ when $j\ge 2$ and $\xi_{\gamma^*}^0,\xi_\gamma^0=0$, we have
	 \begin{gather*}
	 	0\le \xi_\gamma^{j-1}\lesssim \tau_1^\alpha(\tau_1/t_j)^{1+\min\{\gamma,0\}-\alpha},\\
	 	0\le \xi_{\gamma^*}^{j-1}\lesssim \tau_1^\alpha(\tau_1/t_j)^{1+\gamma^*-\alpha},
	 \end{gather*}
	 for $j\ge 1$. Therefore, there exists a positive constant $C_\gamma$, such that $	(\delta_t^{\alpha,*}-\lambda_1)\xi_\gamma^j-\lambda_2\xi_\gamma^{j-1}\ge (\tau_1/t_j)^{\gamma+1}-C_\gamma\tau_1^\alpha(\tau_1/t_j)^{1+\gamma^*},$ and $(\delta_t^{\alpha,*}-\lambda_1)\xi_{\gamma^*}^j-\lambda_2\xi_{\gamma^*}^{j-1}\ge (\tau_1/t_j)^{{\gamma^*}+1}(1-C_\gamma t_j^\alpha),$ for $j\ge 1$.	
	
	  Select a sufficiently small constant $c_1\le (\frac{1}{2C_\gamma})^{\frac{1}{\alpha}}$, and a sufficiently large constant $\bar{b}\ge C_\gamma\max\{2,c_1^{(-1-\gamma^*+\alpha)}, T^{(-1-\gamma^*+\alpha)}\}$, such that for $j\ge 1$, one has
	\begin{align*}	
&\quad (\delta_t^{\alpha,*}-\lambda_1)(\xi_{\gamma}^j+\bar{b}\tau_1^\alpha\xi_{\gamma^*}^j)-\lambda_2(\xi_{\gamma}^{j-1}+\bar{b}\tau_1^\alpha\xi_{\gamma^*}^{j-1}) \\
&\ge(\tau_1/t_j)^{\gamma+1}-C_\gamma\tau_1^\alpha(\tau_1/t_j)^{1+\gamma^*}+\bar{b}\tau_1^\alpha(\tau_1/t_j)^{{\gamma^*}+1}(1-C_\gamma t_j^\alpha)\\
	 	&= (\tau_1/t_j)^{\gamma+1}+(\bar{b}-C_\gamma)\tau_1^\alpha(\tau_1/t_j)^{1+\gamma^*}-\bar{b}C_\gamma\tau_1^\alpha(\tau_1/t_j)^{1+\gamma^*}t_j^\alpha\\
	 	&\ge (\tau_1/t_j)^{\gamma+1}- (\bar{b})^2\begin{cases}
	 		0,\quad &t_j<c_1,\\
	 		\tau_1^{1+\min\{\gamma,0\}},\quad &t_j\ge c_1.
	 	\end{cases}
	 \end{align*}
Introduce the function \( B(t) \) from Lemma \ref{lemma3.1}. Let $c_0=\min\{\frac{c_1}{3},\frac{(2(\lambda_1+\lambda_2)\Gamma(2-\alpha))^{-1/\alpha}}{3}\}$, $\tau\le c_0/2$ and select $t_m\simeq 1$, such that $c_0\ge t_m\ge \frac{1}{2}c_0$ if $T\ge \frac{1}{2}c_0$, and $t_m=T$ otherwise. It holds that $t_m+\frac{5}{4}c_0\le c_1$. The choice of  $c_0$ implies that $\lambda_1\tau^\alpha\le \frac{1}{2\Gamma(2-\alpha)}$. We can obtain a function $B(t)$ satisfying that $B^j=0$ for $ t_j\le t_m$, and
\begin{gather*}
	 0\le B^j\le C_3, \\
		(\delta_t^{\alpha,*}-\lambda_1)B^j-\lambda_2B^{j-1}\ge
		\begin{cases}
			0, \quad t_j^*< t_m+c_0,\\
			1, \quad t_j^*\ge t_m+c_0,
		\end{cases}
\end{gather*}
for $j\ge 1$.

Furthermore, we set $W^j=\xi_\gamma^j+\bar{b}\tau_1^\alpha\xi_{\gamma^*}^j+(\bar{b})^2\tau_1^{1+\min\{\gamma,0\}}B^j$.
When $t_j\le t_m$, one has $B^j=0$, then
\begin{align*}
W^j= \xi_\gamma^j+\bar{b}\tau_1^\alpha\xi_{\gamma^*}^j\le (C_1+C_2\bar{b}T^\alpha)\tau_1^\alpha(\tau_1/t_j)^{1+\gamma^*}.
\end{align*}
When $t_m<t_j\le T$, one has $\frac{1}{2}c_0<t_j\le T$, then
\begin{align*}
	W^j&\le\max\{C_1+C_2\bar{b}T^\alpha,C_3\bar{b}^2\}\left[\tau_1^\alpha(\tau_1/t_j)^{1+\gamma^*}+\tau_1^\alpha(\tau_1/t_j)^{1+\gamma^*}t_j^{1+\gamma^*}\right]\\
	&\le \max\{C_1+C_2\bar{b}T^\alpha,C_3\bar{b}^2\}\left[1+\max\{(c_0/2)^{1+\gamma^*},T^{1+\gamma^*}\}\right]\tau_1^\alpha(\tau_1/t_j)^{1+\gamma^*}.
\end{align*}
	To sum up, $W^j\le C\tau_1^\alpha(\tau_1/t_j)^{1+\gamma^*}$ for $j\ge 1$, where $C$ is independent of mesh sizes. If $t_j\ge c_1$,  $t_j^*=t_j-(1-\sigma)\tau_j\ge t_m+\frac{5}{4}c_0-\frac{1}{2}\tau\ge t_m+c_0$, and then it holds that
	\begin{align*}
		(\delta_t^{\alpha,*}-\lambda_1)B^j-\lambda_2B^{j-1}\ge 1.
	\end{align*}
	Therefore, one has
	\begin{align*}
	(\delta_t^{\alpha,*}-\lambda_1)  W^j-\lambda_2 W^{j-1}\ge (\tau_1/t_j)^{\gamma+1},\ for\ j\ge 1.
	\end{align*}
 By Corollary \ref{corollary3.1}, the proof is completed.\\
	 (ii) If $\gamma\le \alpha-1$, we have $\gamma^*=\min\left\{\gamma,0\right\}-\alpha=\gamma-\alpha<\alpha-1$. Thus, $\xi_\gamma^j$ and $\xi_{\gamma*}^j$ can be constructed as well. The proof can follow the approach demonstrated in (i).
\end{proof}

\section{Convergence and stability analysis}
For any grid functions $\chi$ and $\eta$, we introduce the following notations:
\begin{gather*}
	(\chi,\eta)=h_1h_2\sum_{j=1}^{M_1-1}{\sum_{k=1}^{M_2-1}{\chi_{j,k}\eta_{j,k}}},\
	\|\chi\|=\sqrt{(\chi,\chi)}, \\
	\|\chi\|_\infty=\max\limits_{1\le j\le M_1-1,1\le k\le M_2-1}|\chi_{j,k}|,\ |\chi|_2=\frac{1}{\nu}\sqrt{(L_h\chi,L_h\chi)}.
\end{gather*}
\begin{lemma}
	\label{lemma4.5}
 Let properties P1 and P2 hold. Set  $1\ge\sigma\ge \frac{1}{2}$. For any grid function $\{v^j\}_{j=0}^N$, one has \begin{align}
		\label{eq4.4*}
		(\delta_{t}^{\alpha,*}v^n,v^{n,*})\ge \left[\sigma \|v^n\|+(1-\sigma)\|v^{n-1}\|\right](\delta_t^{\alpha,*}\|v^n\|),\ n=1,...,N.
	\end{align}
\end{lemma}
\begin{proof}
	This is a special case of \cite{ChenSty2025}.
\end{proof}
Define $\Pi_h^0=\{u|u\ is\ a\ grid\ function\ defined\ on\ \Omega'\ and\ u_{i,j}=0\ when\ (x_i,y_j)\in\partial\Omega_h\}$.
\begin{lemma}(\cite{MR4533497},Lemma 3.2)
	\label{lemma*}
	For arbitrary grid function $v\in \Pi_h^0$, one has
	\begin{align*}		
		\|v\|_{\infty} \le C\|v\|^{\frac{1}{2}}\left(\|v\|+|v|_2\right)^{\frac{1}{2}}.
	\end{align*}
\end{lemma}

\begin{lemma}(\cite{MR4533497},Lemma 3.3)
	\label{lemma**}
	For arbitrary grid function $v\in \Pi_h^0$, one has
	\begin{align*}		
		\|v\|_{\infty} \leq  \sqrt{h_1^{-1}h_2^{-1}}\|v\| .
	\end{align*}
\end{lemma}
\begin{lemma}
	\label{lemma***}
	For arbitrary grid function $\{v^j\}_{j=0}^{N}$, if property P1 is possessed, $\rho\le 7/4$ and $\sigma=1-\frac{\alpha}{2}$, it holds that
	\begin{align*}
		\|\delta_{t}^{\alpha,*}v^n\|		
		\le{\frac{48}{11\tau_n^{\alpha}\Gamma(2-\alpha)}} \max_{0\le j\le n}\|v^j\|,\ for\ n=1,2,\dots,N.
	\end{align*}
\end{lemma}
\begin{proof}
	By Lemma \ref{lemma2.1} and property P1, and invoking the Minkowski inequality, one has
		\begin{align*}
		\|\delta_{t}^{\alpha,*}v^n\|		
		&= \|\sum_{j=1}^{n}{g_{n-1,j-1}(v^j-v^{j-1})}\|\\
		&=\|g_{n-1,n-1}v^n-\sum_{j=2}^{n}{(g_{n-1,j-1}-g_{n-1,j-2})v^{j-1}}-g_{n-1,0}v^0\|\\&\le g_{n-1,n-1}\|v^n\|+\sum_{j=2}^{n}{(g_{n-1,j-1}-g_{n-1,j-2})\|v^{j-1}\|}+g_{n-1,0}\|v^0\|\\
		&\le \left[g_{n-1,n-1}+\sum_{j=2}^{n}{(g_{n-1,j-1}-g_{n-1,j-2})}+g_{n-1,0}\right]\max_{0\le j\le n}\|v^j\|\\
		&= 2g_{n-1,n-1}\max_{0\le j\le n}\|v^j\|.
	\end{align*}
Next, we estimate the upper bound of $g_{n-1,n-1}$. In Theorem 2.2 of \cite{MR4270344}, for $n\ge 1$ we obtain
\begin{align*}
	g_{n-1,n-1}&\le \frac{24}{11\tau_n\Gamma(1-\alpha)}\int_{t_{n-1}}^{t_n}{(t_n-s)^{-\alpha}}ds\\
	&= \frac{24}{11\tau_n^{\alpha}\Gamma(2-\alpha)}.
\end{align*}
This proof is completed.
\end{proof}
	For $e^n=u^n-U^n$, where $u^n$ is the solution of \eqref{eq2.1}-\eqref{eq2.3}, and $U^n$ satisfies \eqref{eq2.4}-\eqref{eq2.6}, it holds that
	\begin{align}
		\label{eq4.4}
		&\delta_{t}^{\alpha,*}e_{i,j}^n+L_he_{i,j}^{n,*}=(R_f)_{i,j}^n+R_{i,j}^n,	(x_i,y_j)\in \Omega_h, t_n\in(0,T],\\
		&e_{i,j}^0=0, (x_i,y_j)\in \Omega_h,\\
		&e_{i,j}^n=0, (x_i,y_j)\in \partial\Omega_h,  t_n\in[0,T],
	\end{align}
where $(R_f)_{i,j}^n=f(u_{i,j}^{n-1})+\sigma f'(u_{i,j}^{n-1})(u_{i,j}^n-u_{i,j}^{n-1})-f(U_{i,j}^{n-1})-\sigma f'(U_{i,j}^{n-1})(U_{i,j}^n-U_{i,j}^{n-1})$	 and $R_{i,j}^n=(r_1)_{i,j}^n+(r_2)_{i,j}^n+(r_3)_{i,j}^n$.

According to Lemmas \ref{lemma2.3}, \ref{lemma2.4}, and \ref{lemma2.5}, for $1\le n\le N$, we obtain
\begin{align*}
	\|R^n\|=\|r_1^n+r_2^n+r_3^n\|&\lesssim (\tau_1/t_n)^{\hat\gamma+1}+\tau_n^2t_n^{\alpha-2}+h_1^2+h_2^2 \\
&\simeq (\tau_1/t_n)^{\hat\gamma+1}+\tau_1^{2/r}t_n^{\alpha-2/r}+h_1^2+h_2^2,
\end{align*}
and
\begin{align}
	\max\limits_{1\le n\le N}\|R^n\|\le C(1+L_1^2+L_2^2+T^\alpha):=C_R,
\end{align}
where $1+\hat\gamma=min\{\alpha+1,(3-\alpha)/r\}$. To achieve sharper results, we proceed as follows\cite{MR4085134}:
\begin{align}
	\tau_1^{2/r}t_n^{\alpha-2/r}=
	\begin{cases}
		\tau_1^{\alpha}(\tau_1/t_n)^{2/r-\alpha},\quad &2/r-\alpha<1,\\
		t_n^\alpha(\tau_1/t_n)^{2/r}, \quad & 2/r-\alpha\ge 1.
	\end{cases}
\end{align}
\begin{theorem}
	\label{theorem 4.2}
Let $\sigma=1-\frac{\alpha}{2}$, $\rho\le 7/4$, and properties P1-P2 hold. If $\tau$, $\tau_1$ and $h_1=h_2$ are sufficiently small, set $1\le r < 4/\alpha$. For $\{e^n\}_{n=1}^{N}$, with a small positive constant $\epsilon$, one has
\begin{align}
	\|e^n\|\lesssim\begin{cases}
		\tau_1t_n^{\alpha-1}+(h_1^2+h_2^2)t_n^\alpha,\quad & 1\le r\le 2/(\alpha+1),\\
		\tau_1t_n^{\alpha-1}+\tau_1^{2/r}t_n^{2\alpha-2/r}+(h_1^2+h_2^2)t_n^\alpha, \quad   &2/(\alpha+1)< r <3-\alpha,\\
		\tau_1^{1-\epsilon}t_n^{\epsilon+\alpha-1}+\tau_1^{2/r}t_n^{2\alpha-2/r}+(h_1^2+h_2^2)t_n^\alpha, \quad   &r =3-\alpha,\\
		\tau_1^{\frac{3-\alpha}{r}}t_n^{\alpha-\frac{3-\alpha}{r}}+\tau_1^{2/r}t_n^{2\alpha-2/r}+(h_1^2+h_2^2)t_n^\alpha, \quad &r>3-\alpha,
	\end{cases}
\end{align}
and
\begin{align}
	\label{eq4.6}
		\|U^{n}\|_\infty\le \max_{0\le j\le N}\|u^{j}\|_\infty+1,\ for\ 0\le n\le N.
\end{align}
\end{theorem}
\begin{proof}
	For $n=0$, \eqref{eq4.6} holds obviously. Assuming that \eqref{eq4.6} holds for $n$ = 0, 1, 2, $\cdots, m-1$, then for $1\le k\le m$,  applying Lagrange mean value theorem, the estimate of $R_f$ is given by:
\begin{align*}
	\|(R_f)^k\|&=\|f(u^{k-1})+\sigma f'(u^{k-1})(u^k-u^{k-1})-f(U^{k-1})-\sigma f'(U^{k-1})(U^k-U^{k-1})\|\\
	&\le \|f(u^{k-1})-f(U^{k-1})\|+\sigma\|(f'(u^{k-1})-f'(U^{k-1}))u^k\|+\sigma\|f'(U^{k-1})(u^k-U^{k})\|\\
	&+\sigma\|(f'(u^{k-1})-f'(U^{k-1}))u^{k-1}\|+\sigma\|f'(U^{k-1})(u^{k-1}-U^{k-1})\|\\&\le Q_1(\|e^{k}\|+\|e^{k-1}\|),
\end{align*}
where the positive constant $Q_1$ depends on $\sigma$, $f$ and $\max_{0\le j\le N}\|u^{j}\|_\infty$.
Let $\delta_{t}^{\alpha,*}e^k$ take the inner product with $e^{k,*}$ from \eqref{eq4.4}. Considering $(L_he^{k,*},e^{k,*})\ge 0$, it holds that
\begin{align*}
	(\delta_{t}^{\alpha,*}e^k,e^{k,*})&\le ((R_f)^k+R^k,e^{k,*})\\
	&\le \left[Q_1(\|e^{k}\|+\|e^{k-1}\|)+\|R^k\|\right]\|e^{k,*}\|\\
	&\le  \left[Q_1(\|e^{k}\|+\|e^{k-1}\|)+\|R^k\|\right]\left[\sigma\|e^k\|+(1-\sigma)\|e^{k-1}\|\right].
\end{align*}
Applying Lemma \ref{lemma4.5}, one has
\begin{align}
	\delta_t^{\alpha,*}\|e^k\|&\le Q_1(\|e^{k}\|+\|e^{k-1}\|)+\|R^k\|.
\end{align}
Due to
\begin{align}
	\|R^k\|\lesssim (\tau_1/t_k)^{\hat\gamma+1}+h_1^2+h_2^2+
	\begin{cases}
		\tau_1^{\alpha}(\tau_1/t_k)^{2/r-\alpha},\quad &2/r-\alpha<1,\\
		t_k^\alpha(\tau_1/t_k)^{2/r}, \quad & 2/r-\alpha\ge 1,
	\end{cases}
\end{align}
we set $W_1^0=W_2^0=W_3^0=0$, $\delta_t^{\alpha,*}W_1^k- Q_1(W_1^{k}+W_1^{k-1})=(\tau_1/t_k)^{\hat\gamma+1}$, $\delta_t^{\alpha,*}W_2^k-Q_1(W_2^{k}+W_2^{k-1})=h_1^2+h_2^2$, and
\begin{align*}
	\delta_t^{\alpha,*}W_3^k- Q_1(W_3^{k}+W_3^{k-1})=\begin{cases}
		\tau_1^{\alpha}(\tau_1/t_k)^{2/r-\alpha},\quad &2/r-\alpha<1,\\
		t_k^\alpha(\tau_1/t_k)^{2/r}, \quad & 2/r-\alpha\ge 1.
	\end{cases}
\end{align*}
For any $1\le r\le(3-\alpha)/\alpha$, employing Theorem \ref{lemma3.3}, we can obtain $W_2^k\lesssim (h_1^2+h_2^2)V(k,-1)$,
\begin{equation}
	\label{result2}
	\begin{split}
		W_1^k\lesssim \begin{cases}
			V(k,\hat\gamma-\epsilon),\quad &r= 3-\alpha,\\
			V(k,\hat\gamma),\quad &r\neq 3-\alpha,\\
		\end{cases} \  W_3^k\lesssim\begin{cases}
		\tau_1^\alpha V(k,2/r-\alpha-1),\quad &2/r<\alpha+1,\\
		V(k,2/r-1),\quad &2/r\ge\alpha+1.
		\end{cases}
	\end{split}
\end{equation}
It is worth noting that when $r=3-\alpha$, $\hat\gamma=0$. However, our stability results do not apply to this case, and thus it requires special treatment: $(\tau_1/t_k)^{\hat\gamma+1}\le (\tau_1/t_k)^{\hat\gamma+1-\epsilon}$ with a small $0<\epsilon<1-\alpha$. Since, for any $r>(3-\alpha)/\alpha$, we have both $\hat\gamma<\alpha-1$ and $2/r-\alpha-1<\alpha-1$, the inequalities \eqref{result2} hold for all $1\le r<4/\alpha$. Therefore, the error estimate for $1\le k\le m$ is given by:
\begin{equation*}
	\begin{split}
		\|e^k\|\le K_1\begin{cases}
			V(k,\hat\gamma-\epsilon)+\tau_1^\alpha V(k,2/r-\alpha-1)+(h_1^2+h_2^2)V(k,-1),\quad &r= 3-\alpha,\\
			V(k,\hat\gamma)+\tau_1^\alpha V(k,2/r-\alpha-1)+(h_1^2+h_2^2)V(k,-1),\quad &2/(3-\alpha)\neq2/r<\alpha+1,\\
			V(k,\hat\gamma)+V(k,2/r-1)+(h_1^2+h_2^2)V(k,-1),\quad &2/r\ge\alpha+1.
		\end{cases}
	\end{split}
\end{equation*}
  Furthermore,
\begin{align}
	\|e^k\|\le K_1\begin{cases}
		2\tau_1t_k^{\alpha-1}+(h_1^2+h_2^2)t_k^\alpha,\quad & 1\le r\le 2/(\alpha+1),\\
		\tau_1t_k^{\alpha-1}+\tau_1^{2/r}t_k^{2\alpha-2/r}+(h_1^2+h_2^2)t_k^\alpha, \quad   &2/(\alpha+1)< r <3-\alpha,\\
		\tau_1^{1-\epsilon}t_k^{\epsilon+\alpha-1}+\tau_1^{2/r}t_k^{2\alpha-2/r}+(h_1^2+h_2^2)t_k^\alpha, \quad   &r =3-\alpha,\\
		\tau_1^{\frac{3-\alpha}{r}}t_k^{\alpha-\frac{3-\alpha}{r}}+\tau_1^{2/r}t_k^{2\alpha-2/r}+(h_1^2+h_2^2)t_k^\alpha, \quad &r>3-\alpha.
	\end{cases}
\end{align}
and the maximum value of the error in $L^2$-norm is obtained as:
 \begin{align}
 	\max_{1\le k\le m}\|e^k\|\le K_2( \tau_1^{min\{\alpha,2/r\}}+h_1^2+h_2^2),
 \end{align} where $K_2=K_1\max\{2,T^\alpha+1,T^{2\alpha-2/r}+1\}$.

From \eqref{eq4.4}, we obtain
\begin{align}
	-L_he_{i,j}^{k,*}=\delta_{t}^{\alpha,*}e_{i,j}^k-(R_f)_{i,j}^k-R_{i,j}^k,\ for\ 1\le k\le m.
\end{align}
Taking the discrete $L^2$-norm, we have
\begin{align}
\|L_he^{k,*}\|\le \|\delta_{t}^{\alpha,*}e^k\|+\|(R_f)^k\|+\|R^k\|.
\end{align}
Using $\tau_n\simeq\tau_1^{1/r}t_n^{1-1/r}$, there exists a positive constant $\bar Q$ such that $\tau_n\ge \bar{Q}\tau_1^{1/r}t_n^{1-1/r}\ge \bar{Q}\tau_1$. Then by utilizing the triangle inequality of the $L^2$-norm and Lemma \ref{lemma***}, we have
\begin{align*}
	\nu\left[\sigma|e^{k}|_2-(1-\sigma)|e^{k-1}|_2\right]&\le \|\delta_{t}^{\alpha,*}e^k\|+Q_1(\|e^{k}\|+\|e^{k-1}\|)+\|R^k\|\\
	&\le{\frac{48}{11\tau_n^{\alpha}\Gamma(2-\alpha)}} \max_{1\le j\le m}\|e^j\|+2Q_1\max_{1\le j\le m}\|e^j\|+C_R
	\\
		&\le {\frac{48}{11\Gamma(2-\alpha)}}\bar{Q}^{-\alpha}\tau_1^{-\alpha}\max_{1\le j\le m}\|e^j\|+2Q_1\max_{1\le j\le m}\|e^j\|+C_R
		\\&:=E.
\end{align*}
Since the above relationship holds for $k=1,2,...,m$, that is
\begin{align*}
	\nu\left[\sigma|e^{m}|_2-(1-\sigma)|e^{m-1}|_2\right]&\le E,\\
	\nu\left[\sigma|e^{m-1}|_2-(1-\sigma)|e^{m-2}|_2\right]&\le E,\\
	\dots\\
	\nu\left[\sigma|e^{1}|_2-(1-\sigma)|e^{0}|_2\right]&\le E,
\end{align*}
by eliminating the intermediate terms, we obtain
\begin{align}
	(\frac{\sigma}{1-\sigma})^{m-1}\nu\sigma|e^m|_2\le \left[1+\frac{\sigma}{1-\sigma}+\dots+(\frac{\sigma}{1-\sigma})^{m-1}\right]E.
\end{align}
Due to $\frac{1-\sigma}{\sigma}<1$, it is obvious that
\begin{align}
	\left[1+\frac{\sigma}{1-\sigma}+\dots+(\frac{\sigma}{1-\sigma})^{m-1}\right](\frac{1-\sigma}{\sigma})^{m-1}=\frac{(\frac{1-\sigma}{\sigma})^m-1}{\frac{1-\sigma}{\sigma}-1}\le \frac{1}{1-\frac{1-\sigma}{\sigma}}=\frac{\sigma}{2\sigma-1}.
\end{align}
Therefore, the estimate for the semi-norm of $e_{i,j}^m$ is given by:
\begin{align*}
	|e^{m}|_2&\le \frac{1}{\nu(2\sigma-1)}\left({\frac{48}{11\Gamma(2-\alpha)}}\bar{Q}^{-\alpha}\tau_1^{-\alpha}\max_{1\le j\le m}\|e^j\|+2Q_1\max_{1\le j\le m}\|e^j\|+C_R\right)\\
	&\le (Q_2\tau_1^{-\alpha}+Q_3)\max_{1\le j\le m}\|e^j\|+Q_4,
\end{align*}
where  $Q_2={\frac{1}{\nu(2\sigma-1)}\frac{48}{11\Gamma(2-\alpha)}}\bar{Q}^{-\alpha}$, $Q_3=\frac{2Q_1}{\nu(2\sigma-1)}$, and $Q_4=\frac{C_R}{\nu(2\sigma-1)}$.
$Q_2$, $Q_3$ and $Q_4$ are independent of mesh sizes.

Using Lemma \ref{lemma*}, one has
\begin{align*}
	\|e^{m}\|_\infty&\le C\|e^m\|^{\frac{1}{2}}\left(\|e^m\|+|e^m|_2\right)^{\frac{1}{2}}\\
	&\le C\max_{1\le j\le m}\|e^j\|^{\frac{1}{2}}\left[\max_{1\le j\le m}\|e^j\|+(Q_2\tau_1^{-\alpha}+Q_3)\max_{1\le j\le m}\|e^j\|+Q_4\right]^{\frac{1}{2}}\\
	&\le C(K_2)^{\frac{1}{2}}(\tau_1^{\min\{\alpha,2/r\}}+h_1^2+h_2^2)^{\frac{1}{2}}\left[K_2(1+Q_2\tau_1^{-\alpha}+Q_3)(\tau_1^{\min\{\alpha,2/r\}}+h_1^2+h_2^2)+Q_4\right]^{\frac{1}{2}}\\
	&\le CK_3(\tau_1^{\min\{\alpha,2/r\}}+h_1^2+h_2^2)^{\frac{1}{2}}\left[\tau_1^{-\alpha}(\tau_1^{\min\{\alpha,2/r\}}+h_1^2+h_2^2)+1\right]^{\frac{1}{2}}
\end{align*}
where $K_3=\max\{K_2(T^\alpha+Q_2+Q_3T^\alpha)^{\frac{1}{2}},(K_2Q_4)^{\frac{1}{2}}\}$.

Case A: $h_1^2+h_2^2\le \tau_1^{\min\{\alpha,2/r\}}$.
\begin{align*}
		\|e^{m}\|_\infty&\le \sqrt{2}CK_3(\tau_1^{\min\{\alpha,2/r\}})^{\frac{1}{2}}\left(2\tau_1^{-\alpha}\tau_1^{\min\{\alpha,2/r\}}+1\right)^{\frac{1}{2}}\\
		&\le \sqrt{2}CK_3\left(2\tau_1^{-\alpha}\tau_1^{\min\{2\alpha,4/r\}}+\tau_1^{\min\{\alpha,2/r\}}\right)^{\frac{1}{2}}\\
		&=\sqrt{2}CK_3\left(2\tau_1^{\min\{\alpha,4/r-\alpha\}}+\tau_1^{\min\{\alpha,2/r\}}\right)^{\frac{1}{2}}\\
		&\le \sqrt{6}CK_3\left(\tau_1^{\min\{\alpha,4/r-\alpha,2/r\}}\right)^{\frac{1}{2}}.
\end{align*}
Under the condition $r<4/\alpha$, if $\tau_1$ is sufficiently small, then $\|e^{m}\|_\infty\le 1$.

Case B: $h_1^2+h_2^2>\tau_1^{\min\{\alpha,2/r\}}$. Using Lemma \ref{lemma**}, we have
\begin{align*}
	\|e^{m}\|_\infty&\le \sqrt{h_1^{-1}h_2^{-1}}\|e^m\| \\
	&\le \sqrt{h_1^{-1}h_2^{-1}}\max_{1\le j\le m}\|e^j\|\\
	&\le 2K_2\sqrt{h_1^{-1}h_2^{-1}}(h_1^2+h_2^2).
\end{align*}
If $h_1=h_2$ and $h_1$ is sufficiently small, then $\|e^{m}\|_\infty\le 1$.
Consequently, regardless of the relative magnitudes of $h_1^2+h_2^2$ and $\tau_1^{\min\{\alpha,2/r\}}$, the inequality $\|e^{m}\|_\infty\le 1$ holds if $\tau_1$, $h_1$, and $h_2$ are sufficiently small, which further implies:
\begin{align*}
	\|U^{m}\|_\infty&\le \|u^{m}\|_\infty+\|e^{m}\|_\infty\\
	&\le \|u^{m}\|_\infty+1\\
	&\le \max_{0\le j\le N}\|u^{j}\|_\infty+1.
\end{align*}
	In summary, the inequality \eqref{eq4.6} holds for $n=m$, which finishes the mathematical induction.
\end{proof}

\begin{remark}
	\label{remark 4.2}
	From Theorem \ref{theorem 4.2}, we can further obtain the global convergence: \begin{align*}
		\max_{1\le n\le N}\|e^n\|\lesssim\begin{cases}
			\tau_1^\alpha+h_1^2+h_2^2,\quad&if\ 1\le r< 1/\alpha,\\
			\tau_1^{\min\{\alpha,2/r\}}+h_1^2+h_2^2,\quad&if\ 1/\alpha\le r<(3-\alpha)/\alpha,\\
			\tau_1^{2/r}+h_1^2+h_2^2,\quad&if\ r\ge (3-\alpha)/\alpha.
		\end{cases}
	\end{align*}
	This implies that $\max_{1\le n\le N}\|e^n\|\lesssim \tau_1^{\min\{\alpha,2/r\}}+h_1^2+h_2^2\simeq N^{-\min\{\alpha r,2\}}+h_1^2+h_2^2$.
	
	When $t_n=T$, $\|e^N\|\lesssim
	\begin{cases}
	\tau_1+h_1^2+h_2^2,\quad & if\ 1\le r\le 2/(\alpha+1),\\
\tau_1+\tau_1^{2/r}+h_1^2+h_2^2, \quad  &if\ 2/(\alpha+1)< r <3-\alpha,\\
\tau_1^{1-\epsilon}+\tau_1^{2/r}+h_1^2+h_2^2, \quad  &if\ r =3-\alpha,\\
\tau_1^{2/r}+h_1^2+h_2^2, &if\ r>3-\alpha,
	\end{cases}$\\
	which implies that $\|e^N\|\lesssim \tau_1^{\min\{1,2/r\}}+h_1^2+h_2^2\simeq N^{-\min\{r,2\}}+h_1^2+h_2^2$ if $r\neq 3-\alpha$, and $\|e^N\|\lesssim \tau_1^{\min\{1-\epsilon,2/r\}}+h_1^2+h_2^2\simeq N^{-\min\{(1-\epsilon)r,2\}}+h_1^2+h_2^2$ if $r= 3-\alpha$.\\
	$\mathbf{Notation.}$
	In Theorem \ref{theorem 4.2}, it is required that \(1\le r < 4/\alpha \).
\end{remark}

\begin{lemma}(\cite{MR3904430}, Theorem 3.1)
	\label{lemma4.1}
	Let the properties $P1$, $P3$ and $P4$ hold and $0\le (1-\sigma) <1$. If both $(g^n)_{n=1}^{N}$ and $(\lambda_l)_{l=0}^{N-1}$ are non-negative sequences, and the temporal mesh satisfies that $\tau\le (\sqrt[\alpha]{2\pi_A\Gamma(2-\alpha)\Lambda})^{-1}$, where $\Lambda$ is a constant such that $\Lambda\ge \sum_{l=0}^{N-1}{\lambda_l}$. For any non-negative sequence $(v^k)_{k=0}^{N}$ satisfying $\delta_t^{\alpha,*}(v^n)^2\le \sum_{k=1}^{n}{\lambda_{n-k}(v^{k,*})^2}+v^{k,*}g^n$ for all $1\le n\le N$, one has
	\begin{align*}
		v^n\le 2E_\alpha(2\max\{1,\rho\}\pi_A\Lambda t_n^\alpha)(v^0+\pi_A\Gamma(1-\alpha)\max_{1\le k\le n}\{t_k^\alpha g^k\}),\ for\ 1\le n\le N.
	\end{align*} $E_\alpha$ denote the Mittag Leffler function, i.e., $E_\alpha(\eta)=\sum_{k=0}^{\infty}{\frac{\eta^k}{\Gamma(1+k\alpha)}}$.
\end{lemma}

\begin{lemma}(\cite{MR3904430}, Lemma A.1)
	\label{lemma4.2}
	Suppose that $1\ge\sigma\ge\frac{1}{2}$, and the properties P1 and P2 in Lemma \ref{lemma2.2} hold.  For any grid function $\{v^j\}_{j=0}^{N}\in \Pi_h^0$, it follows that
	\begin{align*}
		(v^{n,*},\delta_{t}^{\alpha,*}v^n)\ge \frac{1}{2} \delta_{t}^{\alpha,*}\|{v^n} \|
		^2\ for\ n\ge 1.
	\end{align*}
\end{lemma}

\begin{theorem}
	Let $U_1$ be the solution of \eqref{eq2.4}-\eqref{eq2.6}, and $U_2$ satisfy the following equation:
	\begin{align}
		\label{eq4.11}
		&\delta_{t}^{\alpha,*}U_{i,j}^n+L_hU_{i,j}^{n,*}=f(U_{i,j}^{n-1})+\sigma f'(U_{i,j}^{n-1})(U_{i,j}^n-U_{i,j}^{n-1}),	(x_i,y_j)\in \Omega_h, t_n\in(0,T]\\
		&U_{i,j}^0=\hat{u}_0(x_i,y_j), (x_i,y_j)\in \Omega_h\\
		&U_{i,j}^n=0, (x_i,y_j)\in \partial\Omega_h,  t_n\in[0,T].
		\label{eq4.13}
	\end{align}
Set $\sigma=1-\frac{\alpha}{2}$. Assuming that properties P1-P4 hold, $\tau\le \left(\sqrt[\alpha]{\frac{8}{\alpha}Q_5\pi_A\Gamma(2-\alpha)}\right)^{-1}$, and the mesh satisfies the conditions of Theorem \ref{theorem 4.2}, then it holds that
	$\|U_1^n-U_2^n\|\le C\|u_0-\hat{u}_0\|$ for $0\le n\le N$. $Q_5$ is mentioned in \eqref{eq4.14}.
\end{theorem}
\begin{proof}
	Subtracting \eqref{eq4.11} from \eqref{eq2.4}, and
	taking the inner product of both sides with $\hat{e}^{n,*}$, one has
	\begin{align}
		(\delta_t^{\alpha,*}\hat{e}^n,\hat{e}^{n,*})+(L_h\hat{e}^{n,*},\hat{e}^{n,*})=(E_f^n,\hat{e}^{n,*}),\ for\ n\ge 1,
	\end{align}
	where $\hat{e}^n=U_1^n-U_2^n$ and $\|E_f^n\|=\|f(U_1^{n-1})+\sigma f'(U_1^{n-1})(U_1^n-U_1^{n-1})-f(U_2^{n-1})-\sigma f'(U_2^{n-1})(U_2^n-U_2^{n-1})\|$. In Theorem \ref{theorem 4.2}, we have proved that $U_1$, and $U_2$ are bounded, i.e., $\|U_1^n\|_\infty\le C$, and $\|U_2^n\|_\infty\le C$ for $0\le n\le N$. Applying the Lagrange mean value theorem again, for $1\le n\le N$ we obtain
	\begin{equation}
			\begin{split}
			\|E_f^n\|&\le \|f(U_1^{n-1})-f(U_2^{n-1})\|+\sigma\|(f'(U_1^{n-1})-f'(U_2^{n-1}))U_1^n\|+\sigma\|f'(U_2^{n-1})(U_1^n-U_2^{n})\|\\
			\label{eq4.14}
			&+\sigma\|(f'(U_1^{n-1})-f'(U_2^{n-1}))U_1^{n-1}\|+\sigma\|f'(U_2^{n-1})(U_1^{n-1}-U_2^{n-1})\|\\
			&\le Q_5(\|\hat{e}^n\|+\|\hat{e}^{n-1}\|).
		\end{split}
	\end{equation}
		 Based on Lemma \ref{lemma4.2}, and the fact that $(L_h\hat{e}^{n,*},\hat{e}^{n,*})\ge 0$, we have
	 \begin{align*}
	 	\frac{1}{2}\delta_t^{\alpha,*}\|\hat{e}^n\|^2&\le Q_5(\|\hat{e}^n\|+\|\hat{e}^{n-1}\|)\|\hat{e}^{n,*}\|\\
	 	&\le Q_5\frac{1}{1-\sigma}\left[\sigma\|\hat{e}^n\|+(1-\sigma)\|\hat{e}^{n-1}\|\right]\|\hat{e}^{n,*}\|\\
	 	&\le \frac{2}{\alpha}Q_5\left[\sigma\|\hat{e}^{n}\|+(1-\sigma)\|\hat{e}^{n-1}\|\right]^2\ for\ n\ge 1.
	 \end{align*}
	 By setting $v^k=\|\hat{e}^k\|$, we can obtain an inequality as follows:
	\begin{align}
		\delta_t^{\alpha,*}(v^n)^2\le \frac{4}{\alpha}Q_5 (v^{n,*})^2\ for\ n\ge 1.
	\end{align}
	By Lemma \ref{lemma4.1}, it holds that $\|\hat{e}^n\|=v^n\le 2E_\alpha(\frac{8}{\alpha}Q_5\max\{1,\rho\}\pi_A t_n^\alpha)\|\hat{e}^0\|$, for $1\le n\le N$.
\end{proof}

\section{Numerical examples}
In this section, we compute the convergence orders for two numerical examples with $\sigma=1-\frac{\alpha}{2}$ on a standard graded mesh. The standard graded temporal mesh $t_n=T(n/N)^r$ with $r\ge 1$ is a special case of quasi-graded meshes. For a standard graded mesh, the following relationship holds(\cite{MR3936261},Lemma 6): $\rho_{j-1}\ge \rho_j\ge 1$ for $j\ge 2$. Therefore, properties P1-P4 hold. To ensure accuracy, we use the Gauss-Kronrod quadrature to calculate the results for $a_{k,j}$ and $b_{k,j}$ in Lemma \ref{lemma2.1}. Denote $\|e^N\|$, $\max_{1\le n\le N}\|e^n\|$ by $E_L$ and $E_G$, respectively.
  \begin{example}\label{ex1}
  Firstly, consider a two-dimensional nonlinear equation with $\Omega=(0,1)\times(0,1)$, $T=0.5$, $\nu=0.1$. The specific equation is as follows:
 	\begin{align}
 		& D_{t}^{\alpha }u-\nu\Delta u=u-u^3+f(\bm{x},t),\ (\bm{x},t)\in \Omega\times \left(0,T\right],\\
 		&u(\bm{x},0)=u_0(x),\ \bm{x}\in \Omega,\\
 		&u(\bm{x},t)=0,\ (\bm{x},t)\in \partial\Omega\times \left[0,T\right].
 	\end{align}
 	 We let the exact solution be $u=t^\alpha xy(1-x)(1-y)$, and the force term $f(\bm{x},t)$ can be calculated.
 \end{example}		
Tables \ref{Tab:ex1}--\ref{Tab:ex13} present the local and global errors and convergence rates for Example \ref{ex1} with different choices of $r$, which show that our analysis is sharp. From these results, one can see the local error gets the optimal convergence order $2$ under milder grading parameter $r=2$, while for global error, under the same choice $r=2$, its convergence order is only $2\alpha$.
 	\begin{table}[H]
 	\centering
 	\caption{Local and global errors and convergence rates with $r=1$ for Example \ref{ex1}}
 	\begin{tabular}{cccccccc}
 		\toprule
 		$M_1=M_2=25$	& $N$ & 64 & 128 & 256 & 512 & 1024&expected order \\
 		\midrule
 		\multirow{6}{*}{$E_L$}
 		& $\alpha = 0.3$ & 1.1289e-05 & 5.5701e-06 & 2.7549e-06 &1.3652e-06 & 6.7759e-07 \\
 		&  &  & 1.0192 & 1.0157 & 1.0129 & 1.0106 &1\\
 		& $\alpha = 0.5$ & 1.2444e-05 & 6.2554e-06 & 3.1398e-06 & 1.5742e-06 & 7.8864e-07 \\
 		&  &  & 0.9923 & 0.9945 & 0.9960 & 0.9972 &1\\
 		& $\alpha = 0.7$  &1.0889e-05&5.4981e-06&2.7663e-06&1.3887e-06&6.9612e-07\\
 		&&&0.9859&0.9910&0.9942&0.9963&1\\
 		\midrule
 		\multirow{6}{*}{$E_G$}
 		& $\alpha = 0.3$ & 4.0699e-04 & 3.6508e-04 & 3.2066e-04 & 2.7717e-04 & 2.3661e-04 \\
 		&  &  & 0.1568 & 0.1872 & 0.2103 & 0.2282 &0.3\\
 		& $\alpha = 0.5$ & 2.2884e-04 & 1.7075e-04 & 1.2536e-04 & 9.1012e-05 & 6.5564e-05 \\
 		&  &  & 0.4225 & 0.4458 & 0.4619 & 0.4731 &0.5\\
 		& $\alpha = 0.7$ & 7.5376e-05&4.7568e-05&2.9733e-05&1.8475e-05&1.1439e-05 \\
 		&  &  &0.6641&0.6780&0.6864&0.6917&0.7 \\
 		\bottomrule
 	\end{tabular}\label{Tab:ex1}
 \end{table}

		\begin{table}[ht!]
			\centering
			\caption{Local and global errors and convergence rates with $r=2$ for Example \ref{ex1}}
			\begin{tabular}{cccccccc}
				\toprule
	$M_1=M_2=25$	& $N$ & 64 & 128 & 256 & 512 & 1024&expected order \\
				\midrule
				\multirow{6}{*}{$E_L$}
				& $\alpha = 0.3$ & 9.2114e-08 & 2.0588e-08 & 4.7006e-09 & 1.0968e-09 & 2.6082e-10 \\
				&  &  & 2.1616 & 2.1309 & 2.0995 & 2.0723 &2\\
				& $\alpha = 0.5$ & 1.8993e-07 & 4.7388e-08 & 1.1815e-08 & 2.9480e-09 & 7.3626e-10 \\
				&  &  & 2.0029 & 2.0039 & 2.0028 &2.0015 &2\\
				& $\alpha = 0.7$  &3.5427e-07&9.4715e-08&2.4918e-08&6.4808e-09&1.6713e-09\\
				 &&&1.9032&1.9264&1.9429&1.9552&2\\
				\midrule
				\multirow{6}{*}{$E_G$}
				& $\alpha = 0.3$ & 1.6779e-04 & 1.1604e-04 & 7.8963e-05 & 5.3169e-05 & 3.5554e-05 \\
				&  &  & 0.5320 & 0.5554 & 0.5706 & 0.5806 &0.6\\
				& $\alpha = 0.5$ & 3.3526e-05 & 1.6952e-05 & 8.5238e-06 & 4.2739e-06 & 2.1399e-06 \\
				&  &  & 0.9838 &0.9919 &0.9959 & 0.9980 &1.0\\
				& $\alpha = 0.7$ & 4.5169e-06&1.7218e-06&6.5389e-07&2.4799e-07&9.4001e-08 \\
				&  &  & 1.3914&1.3968&1.3988&1.3995&1.4 \\
				\bottomrule
			\end{tabular}\label{Tab:ex12}
		\end{table}

				\begin{table}[ht!]
			\centering
			\caption{Local and global errors and convergence rates with $r=2/\alpha$ for Example \ref{ex1}}
			\begin{tabular}{cccccccc}
				\toprule
				$M_1=M_2=25$	& $N$ & 64 & 128 & 256 & 512 & 1024&expected order \\
				\midrule
				\multirow{6}{*}{$E_L$}
				& $\alpha = 0.3$ & 1.5049e-06 & 4.1242e-07 & 1.0947e-07 & 2.8441e-08 & 7.2869e-09 \\
				&  &  & 1.8675 & 1.9136 & 1.9445 & 1.9646 &2\\
				& $\alpha = 0.5$ & 6.8032e-07 & 1.9328e-07 & 5.2776e-08 & 1.4023e-08 & 3.6571e-09\\
				&  &  & 1.8156 & 1.8727 & 1.9121 & 1.9390 &2\\
				& $\alpha = 0.7$  &8.2232e-08&3.3182e-08&1.1382e-08&3.5367e-09&1.0332e-09\\
				&&&1.3093&1.5437&1.6862&1.7753&2\\
				\midrule
				\multirow{6}{*}{$E_G$}
				& $a = 0.3$ & 1.5611e-06 & 4.1242e-07 & 1.0947e-07 & 2.8441e-08 & 7.2869e-09 \\
				&  &  & 1.9203 & 1.9136 & 1.9445 & 1.9646 &2\\
				& $\alpha = 0.5$ & 1.2189e-06 & 3.0802e-07 & 7.7260e-08 & 1.9331e-08 & 4.8337e-09 \\
				&  &  & 1.9845 & 1.9952 & 1.9988 & 1.9997 &2\\
				& $\alpha = 0.7$ & 8.4567e-07&2.1539e-07&5.4163e-08&1.3562e-08&3.3919e-09 \\
				&  &  & 1.9732&1.9916&1.9977&1.9994&2 \\
				\bottomrule
			\end{tabular}\label{Tab:ex13}
		\end{table}
\begin{example}\label{ex2}
Secondly, we consider the following nonlinear problem without an exact solution in the domain $\Omega=(0,1)\times(0,1)$ and $T=0.5$:
	\begin{align}
		& D_{t}^{\alpha }u-\nu\Delta u=u(1-u)+f(\bm{x},t),\ (\bm{x},t)\in \Omega\times \left(0,T\right],\\
		&u(\bm{x},0)=sin(\pi x)sin(\pi y),\ \bm{x}\in \Omega,\\
		&u(\bm{x},t)=0,\ (\bm{x},t)\in \partial\Omega\times \left[0,T\right],
	\end{align}
	where the diffusion coefficient $\nu$ is 0.3.
\end{example}
As the exact solution is unknown for the numerical example, we employ the two-mesh method\cite{MR1750671} to compute the local error and the convergence order in the temporal direction. The local error is given by:
\begin{align*}
E_L=\|U^N-V^{2N}\|.
\end{align*}
Here, $V^{n}$ denotes the numerical solution obtained on the second temporal grid $t_n=T(\frac{n}{2N})^r$.

Tables \ref{tab:ex2}--\ref{tab:ex23} present the local errors and convergence rates for Example \ref{ex2} with different choices of $r$. From which one can see the convergence order of local error is $O(N^{-\min\{r, 2\}})$,
which further illustrate that our analysis is sharp.
	\begin{table}[h]
	\centering
	\caption{Local errors and convergence rates with $r=1$ for Example \ref{ex2}}
	\begin{tabular}{cccccccc}
		\toprule
		$M_1=M_2=25$	& $N$ & 32 & 64 & 128 & 256 & 512&expected order \\
		\midrule
		\multirow{6}{*}{$E_L$}
		& $\alpha = 0.3$ & 1.0119e-04 & 4.9593e-05 & 2.4344e-05 & 1.1971e-05 & 5.8990e-06 \\
		&  &  & 1.0289 & 1.0266 & 1.0240 & 1.0210 &1\\
		& $\alpha = 0.5$ & 1.0303e-04 & 5.1394e-05 & 2.5677e-05 & 1.2861e-05 & 6.4536e-06 \\
		&  &  & 1.0034 & 1.0011 & 0.9975 & 0.9948 &1\\
		& $\alpha = 0.7$  &7.0050e-05&3.8347e-05&2.0242e-05&1.0467e-05&5.3441e-06\\
		&&&0.8693&0.9217&0.9516&0.9698&1\\
		\bottomrule
	\end{tabular}\label{tab:ex2}
\end{table}

\begin{table}[h]
	\centering
	\caption{Local errors and convergence rates with $r=2$ for Example \ref{ex2}}
	\begin{tabular}{cccccccc}
		\toprule
		$M_1=M_2=25$	& $N$ & 32 & 64 & 128 & 256 & 512&expected order \\
		\midrule
		\multirow{6}{*}{$E_L$}
		& $\alpha= 0.3$ & 1.2807e-06 & 2.9426e-07 & 7.0058e-08 & 1.7236e-08 & 4.3390e-09 \\
		&  &  & 2.1218 & 2.0705 & 2.0231 & 1.9900 &2\\
		& $\alpha = 0.5$ & 5.4414e-06 & 1.4114e-06 & 3.6036e-07 & 9.0765e-08 & 2.2662e-08 \\
		&  &  & 1.9469 & 1.9696 & 1.9892 & 2.0019 &2\\
		& $\alpha = 0.7$  &2.1078e-05&5.4404e-06&1.3928e-06&3.5403e-07&8.9523e-08\\
		&&&1.9539&1.9657&1.9761&1.9835&2\\
		\bottomrule
	\end{tabular}\label{tab:ex22}
\end{table}

\begin{table}[h]
	\centering
	\caption{Local errors and convergence rates with $r=2/\alpha$ for Example \ref{ex2}}
	\begin{tabular}{cccccccc}
		\toprule
		$M_1=M_2=25$	& $N$ & 32 & 64 & 128 & 256 & 512&expected order \\
		\midrule
		\multirow{6}{*}{$E_L$}
		& $\alpha = 0.3$ & 3.7638e-05 & 9.9018e-06 & 2.5602e-06 & 6.5394e-07 & 1.6570e-07 \\
		&  &  & 1.9264 & 1.9514 & 1.9690 & 1.9806 &2\\
		& $\alpha = 0.5$ & 4.2338e-05 & 1.1227e-05 & 2.9306e-06 & 7.5532e-07 & 1.9288e-07 \\
		&  &  & 1.9150 & 1.9377 & 1.9560 & 1.9694 &2\\
		& $\alpha = 0.7$  &4.9098e-05&1.3088e-05&3.4544e-06&9.0249e-07&2.3374e-07\\
		&&&1.9074&1.9218&1.9364&1.9490&2\\
		\bottomrule
	\end{tabular}\label{tab:ex23}
\end{table}

	
	\section*{Acknowledgements}
	The research is supported in part by Natural Science Foundation of Shandong Province under Grant ZR2023MA077, the National Natural Science Foundation of China (Nos.  11801026 and 12171141), and Fundamental Research Funds for the Central Universities (No. 202264006).
	
	\section*{Declaration of competing interest}
	The authors declare no competing interest.
	
	
	
	\bibliographystyle{plain}
	\bibliography{Alikref}
	
	
	%
	%
	%
\end{document}